\documentclass[11pt,dvipsnames]{article}
\usepackage{color} 
\usepackage{amssymb,amsmath,mathrsfs,amsthm}
\usepackage{enumitem}
\usepackage{mathtools}
\usepackage{geometry}
\usepackage{amsfonts}
  \chardef\forshowkeys=0
  \chardef\refcheck=0
  \chardef\showllabel=0
\ifnum\forshowkeys=1
   \usepackage[notref,notcite,color]{showkeys}
\fi
\usepackage[unicode,breaklinks=true,colorlinks=true,linkcolor=blue,urlcolor=blue,citecolor=blue]{hyperref}

\ifnum\refcheck=1
\usepackage{yfonts}\usepackage{refcheck}
\fi

\usepackage{esint,comment}

\usepackage[dvips]{graphicx}
\allowdisplaybreaks
\newtheorem{thm}{Theorem}
\newtheorem{cor}[thm]{Corollary}

 \newtheorem{lemma}[thm]{Lemma}

\theoremstyle{definition}

\numberwithin{equation}{section}

\def\indeq{\quad{}} 

\def\comma{ {\rm ,\qquad{}} }

\def\colb{\color{black}}

\def \no#1#2#3 {{\bf #1} (#3), #2.}
\def \eds#1#2#3 {#1, #2, #3.}

\def\d{{\rm d}}

\def\:{{\colon}}

\def\be#1{\begin{equation}\label{#1}}
\def\ee{\end{equation}}

\def\<{\langle}
\def\>{\rangle}
\def\coloneqq{:=}

\newcommand{\na}{\nabla}

\newcommand{\lec}{\lesssim}
\newcommand{\les}{\lesssim}
\newcommand{\bs}{\begin{split}}
\newcommand{\essss}{\end{split}}

\renewcommand{\div}{\operatorname{div}}
 
\newcommand{\eqnb}{\begin{equation}}
\newcommand{\eqne}{\end{equation}}

\renewcommand{\ee}{\mathrm{e}}

\newcommand{\p}{\partial}

\newcommand{\RR}{\mathbb{R}}

\renewcommand{\d}{\mathrm{d}}

\newcommand{\supp}{\operatorname{supp}}

\newcommand\blfootnote[1]{%
  \begingroup
  \renewcommand\thefootnote{}\footnote{#1}%
  \addtocounter{footnote}{-1}%
  \endgroup
}

\begin{document}
\title{Local-in-time existence of free-surface 3D Euler flow with $H^{2+\delta}$ initial vorticity in a neighborhood of the free boundary}
\author{I. Kukavica, W. S. O\.za\'nski}
\date{\vspace{-5ex}}
\maketitle
\blfootnote{I.~Kukavica: Department of Mathematics, University of Southern California, Los Angeles, CA 90089, USA, email: kukavica@usc.edu\\
W.~S.~O\.za\'nski: Department of Mathematics, University of Southern California, Los Angeles, CA 90089, USA, email: ozanski@usc.edu\\ 
I.~Kukavica was supported in part by the
NSF grant DMS-1907992. W.~S.~O\.za\'nski was supported in part by the Simons Foundation. }

\begin{abstract}
We consider the three-dimensional Euler equations in a domain with a free boundary with no surface tension. We assume that $u_0 \in H^{2.5+\delta }$ is such that $\mathrm{curl}\,u_0 \in H^{2+\delta }$ in an arbitrarily small neighborhood of the free boundary, and we use Lagrangian approach to derive an a~priori estimate that can be used to prove local-in-time existence and uniqueness of solutions under the Rayleigh-Taylor stability condition.
\end{abstract}


\section{Introduction}
In this note we address the local existence of solutions to the free-surface Euler equations
\eqnb\label{euler}
\begin{split}
\p_t u + (u\cdot \na ) u +\na p &=0 \\
\mathrm{div}\, u &= 0
\end{split}
\eqne
in $\Omega(t)$, where $\Omega (t)$ is the region that is periodic in $x_1,x_2$, and $x_3$ lies between $\Gamma_0 \coloneqq \{ x_3=0 \}$ and a free surface $\Gamma_1 (t)$ such that $\Gamma_1 (0)\coloneqq \{ x_3=1 \}$. The Euler equations \eqref{euler} are supplemented with an initial condition $u(0)=u_0$.
We consider the impermeable boundary condition on the fixed bottom boundary $\Gamma_0$ and no surface tension on the free boundary, that is 
\eqnb\label{bcs}
\begin{split}
u_3&=0\qquad \text{ on }\Gamma_0,\\
p&=0 \qquad \text{ on }\Gamma_1.
\end{split}
\eqne
Furthermore the initial pressure is assumed to satisfy the Rayleigh-Taylor condition
\eqnb\label{rayleigh_taylor}
\p_{x_3} p (x,0) \leq -b <0 \qquad \text{ for } x\in \Gamma_1 (0),
\eqne
where $b>0$ is a constant.

The problem of local existence of solutions to the free boundary Euler equations
has been initially considered in \cite{Sh,Shn,N,Y1,Y2} under the
assumption of irrotationality of the initial data, i.e., with
$\mathrm{curl}\,u_0=0$. The local 
existence of solutions in such case with $u_0$ from a Sobolev space was
established by Wu in \cite{W1,W2}. Ebin showed in \cite{E} that for
general data in a Sobolev space, the problem is unstable without
assuming
the Rayleigh-Taylor sign condition \eqref{rayleigh_taylor}, which in
essence requires that next to the free
boundary the pressure in
the fluid is higher than the pressure of air.

With the Rayleigh-Taylor condition, the a~priori bounds for the existence of solutions for initial data
in a Sobolev space were provided in \cite{ChL}. New energy estimates in the Lagrangian coordinates along with the construction of solutions were provided in~\cite{CS1} (see also \cite{CS2}).
We refer the reader to \cite{ABZ1,B,CL,DE1,DE2,DK,KT1,KT2,KT3,S,T,ZZ} for other works concerned with local or global existence results of the Euler equations with or without surface tension, and to \cite{ABZ2,AD,AM1,AM2,EL,GMS,HIT,I,IT,IK,IP,KT3,L,Lin1,Lin2,MR,OT,P,W3} for other related works on the Euler equations with a free-evolving boundary.

In this paper, we are concerned with the problem of local existence of solutions, and we aim to impose minimal regularity assumptions on the initial data $u_0$ that gives local-in-time existence of solutions in the Lagrangian setting of the problem.   It is well-known that the threshold of regularity for the
Euler equations in ${\mathbb R}^{n}$ (\cite{KP}) or a fixed domain is $u_0\in H^{2.5+\delta}$,
where $\delta>0$ is arbitrarily small.
In the case of the domain with a free boundary, 
this has recently been proven in \cite{WZZZ} in the Eulerian candidates
(cf.~\cite{SZ1,SZ2} for the local existence in $H^{3}$). However, the
same result in the Lagrangian coordinates is not known. One of the main difficulties is the fact that it is not clear whether an assumption that $f,g\in H^{2.5+\delta } $ implies that the same is true of $f\circ g $ (the composition of $f$ and $g$) when $\delta$ is not an integer.
The main result of this paper is to obtain a~priori estimates for the local existence
in $H^{2.5+\delta}$ with an additional regularity assumption on initial vorticity $\omega_0\coloneqq \mathrm{curl}\, u_0$ in an arbitrarily small neighborhood of the free boundary.

We note that a related result was obtained in \cite{KTVW} in the 2D case, but the
coordinates used in \cite{KTVW} are not Lagrangian; they are in a sense a concatenation
of Eulerian and Lagrangian variables. 
Moreover, the proof in~\cite{KTVW} uses in an essential way the preservation of Lagrangian vorticity, which is a property that does not hold
in the 3D case. Another paper \cite{KT4} considers the problem
in ALE coordinates (cf.~\cite{MC} for the ALE coordinates in the fluid-structure interaction problem), but the proof requires an additional assumption on the initial data, due to additional regularity required by the Lagrangian variable on the boundary; for instance, the present paper shows that the conditions of Theorem~\ref{thm_main} (see below) suffice.

The main feature of the present paper is that the particle map preserves additional ($H^{3+\delta}$) regularity for a short time in a small neighborhood of the boundary, a feature
not available in the Eulerian approach.
The proof of our main result, which we state in Theorem~\ref{thm_main} below, after a brief introduction of the Lagrangian coordinates, is direct and short. It is inspired by earlier works
\cite{KT2,KTV} and previously by \cite{ChL,CS1,CS2}. 
Another new idea of our approach is a localization of the tangential
estimates (i.e., estimates concerned with differential operators with respect to $x_1$, $x_2$ only) which, due to the additional regularity of the particle map, can be  performed in a domain close to the boundary. 
Moreover, we formulate our a priori estimates in terms of three quantities (see \eqref{big_3}) that control the system, and we make use of the fact that all our estimates, except for the tangential estimates, are not of evolution type. Namely, they are static. 

We note that all the results in the paper also hold in the 2D case with all the Sobolev exponents reduced by~$1/2$.

In the next section we introduce our notation regarding the Euler equations in the Lagrangian coordinates, state the main result, and recall some known facts and inequalities. The proof of Theorem~\ref{thm_main} is presented in Section~\ref{sec_proof_main_result}, where we first give a heuristic argument about our strategy. The proof is then divided into several parts: The div-curl-type estimates are presented in Section~\ref{sec_div_curl}, the pressure estimates are discussed in Sections~\ref{sec_pressure}--\ref{sec_pressure_time_deriv}, and the tangential estimates are presented in Section~\ref{sec_tangential}. The proof of the theorem is concluded with the final estimates in Section~\ref{sec_final}.

\section{Preliminaries}\label{sec_prelims}
\subsection{Lagrangian setting of the Euler equations and the main theorem}
We use the summation convention of repeated indices. We denote the time derivative by $\p_t$, and a derivative with respect to $x_j$ by $\p_j$.
We denote by $v(x,t)= (v_1,v_2,v_3)$ the velocity field in Lagrangian coordinates and by $q(x,t)$ the pressure function. The Euler equations then become
\eqnb\label{euler_lagr}
\begin{split}
\p_t v_i &= - a_{ki} \p_k q,\qquad i=1,2,3,\\
a_{ik} \p_i v_k &=0
\end{split}
\eqne
in $\Omega \times (0,T)$, where $\Omega \coloneqq \Omega (0) = \mathbb{T}^2 \times (0,1)$, and $a_{ik}$ denotes the $(i,k)$-th entry 
of the matrix $a=(\na \eta )^{-1}$. 
Here $\eta $ stands for the particle map, i.e., the solution of the system
\eqnb\label{eta_def}
\begin{split}
\p_t \eta (x,t) &= v(x,t) \\
\eta (x,0) &= x
\end{split} 
\eqne
in $\Omega \times [0,T)$. 
(Note that the Lagrangian variable is denoted by $x$.)
Due to the incompressibility condition in \eqref{euler_lagr}, we have that $\mathrm{det}\, \na \eta =1$ for all times, which shows that
$a$ is the corresponding cofactor matrix. Therefore,
\eqnb\label{a_exact_form}
a_{ij} =\frac12 \epsilon_{imn} \epsilon_{jkl} \partial_{m}\eta_k\partial_{n}\eta_l
,
\eqne
where $\epsilon_{ijk}$ denotes the permutation symbol.
As for the boundary conditions \eqref{bcs}, in Lagrangian coordinates the impermeable condition for $u$ at the stationary bottom boundary $\Gamma_0$ becomes
\eqnb\label{noslip_at_bottom}
v_3 =0 \qquad \text{ on }\Gamma_0,
\eqne
while the zero surface tension condition at the top boundary $\Gamma_1 \coloneqq \Gamma_1(0) = \{ x_3 =1 \}$ reads
\eqnb\label{no_q_ontop}
q=0 \qquad \text{ on } \Gamma_1 \times (0,T).
\eqne
Note that, in Lagrangian coordinates, $\Gamma_1$ does not depend on time. 

Let us consider a localization of $u_0$ given by $\chi u_0$, where $\chi \equiv \chi (x_3) \in C^\infty (\RR ; [0,1])$ is such that $\chi (x_3) =1$ in a neighborhood of $\Gamma_1 (0)$ and $\chi( x_3 ) =0$ outside of a larger neighborhood. 
The following is our main result establishing a~priori estimates for
the local existence of solutions of the free boundary Euler equations
in the Lagrangian formulation.

\begin{thm}\label{thm_main}
Let $\delta >0$. 
Assume that $(v,q,a,\eta)$ is a $C^{\infty}$ solution of the Euler system in the Lagrangian setting
\eqref{euler_lagr}--\eqref{no_q_ontop}, and assume that $u_0$ satisfies the 
Rayleigh-Taylor condition \eqref{rayleigh_taylor}.
Then there exists a time $T>0$ depending on
$\Vert v_0\Vert_{2.5+\delta}$ 
and
$\Vert \chi (\mathrm{curl}\, u_0 ) \Vert_{2+\delta}$
such that
the norms
$\sup_t \Vert v\Vert_{2.5+\delta}$,
$\sup_t \Vert q\Vert_{2.5+\delta}$,
$\sup_t \Vert q \chi\Vert_{3+\delta}$,
and
$\sup_t \Vert \chi\eta\Vert_{3+\delta}$
on $[0,T]$ are bounded from above by a constant depending only on $\Vert v_0\Vert_{2.5+\delta}$
and
$\Vert \chi (\mathrm{curl}\, u_0 ) \Vert_{2+\delta}$.
\end{thm}
\colb

The rest of the paper is devoted to the proof of this theorem.

\subsection{Product and commutator estimates}
We use the standard notions of Lebesgue spaces, $L^p$, 
and Sobolev spaces, $W^{k,p}$, $H^s$, and we reserve the notation $\| \cdot \|_s \coloneqq \| \cdot \|_{H^s (\Omega )}$ for the $H^s$ norm. 
We recall the multiplicative Sobolev inequality
\eqnb\label{kpv_est}
\| fg \|_{s} \lec \| f \|_{{s }} \| g \|_{{1.5+\delta }}
   \comma s\in [0,1.5+\delta ].
\eqne
We shall also use the commutator estimates \eqnb\label{kpv1}
\| J(fg) - f Jg \|_{L^2} \lec \| f \|_{W^{s,p_1}} \| g \|_{L^{p_2}} + \| f \|_{W^{1,q_1}} \| g \|_{W^{s-1,q_2}}
\eqne
for $s\geq 1$ and $\frac{1}{p_1}+\frac1{p_2}= \frac{1}{q_1}+\frac1{q_2}= \frac12$, and
\eqnb\label{kpv3}
\| J(fg) - f Jg - g Jf \|_{L^p} \lec \| f \|_{W^{1,p_1}} \| g \|_{W^{s-1,p_2}}+ \| f \|_{W^{s-1,q_1}} \| g \|_{W^{1,q_2}}
\eqne
for $s\geq 1 $, $\frac{1}{p_1}+\frac1{p_2}=\frac{1}{q_1}+\frac1{q_2}= \frac1p$, $p\in (1,p_1)$, and $p_2,q_1,q_2<\infty$, 
where $J$ is a nonhomogeneous differential operator in $(x_1,x_2)$ of order $s\geq 0$. 
We refer the reader to \cite{KP,Li} for the proofs. We set
\[
\Lambda \coloneqq (1-\Delta_2 )^{\frac12}
\]
and
\[
S\coloneqq \Lambda^{\frac52 + {\delta}},
\]
where $\Delta_2$ denotes the Laplacian in $(x_1,x_2)$.

\subsection{Properties of the particle map $\eta$ and the cofactor matrix~$a$}
Note that applying the product estimate \eqref{kpv_est} to the representation formula \eqref{a_exact_form} for $a$, we get
\eqnb\label{a_in_H1.5}
\| a \|_{{1.5+\delta }} \lec \| \eta \|_{{2.5+\delta }}^2.
\eqne
Moreover, by writing $\chi \na \eta = \na (\chi \eta) - \eta\na \chi
$, where $\chi$ is as above,
we obtain
\eqnb\label{achi_in_H2}
\| 
\chi^2 a \|_{{2+\delta }} \lec \| \chi \eta \|_{{3+\delta }}^2+\|  \eta \|_{{2.5+\delta }}^2
.
\eqne
Also, we have
\eqnb\label{at_in_H1.5}
\| a_t  \|_{{r}} \lec \| \na v \|_{{r}} \qquad \text{ for }r\in [0,1.5+\delta ]
\eqne
and
$
\| \p_{tt} a \|_{r} \lec \| \na v \|_{1.5+\delta } \| \na v \|_r + \| \na \p_t v \|_r
$,
from where
\eqnb\label{att_in_Hr}
\| \p_{tt} a \|_{r} \lec \| \na v \|_{1.5+\delta } \| \na v \|_r + \|  a \|_{1.5+\delta } \| q \|_{2+r }
.
\eqne
Finally, we recall the Brezis-Bourgonion inequality
\eqnb\label{harmonic_est}
\| f \|_{s} \lec \| f \|_{L^2 } + \| \mathrm{curl}\, f \|_{{s-1 }} + \| \mathrm{div}\, f \|_{{s-1  }} + \| \na_2 f_3   \|_{H^{s-0.5} (\partial \Omega )},
\eqne
cf.~\cite{BB}.

\section{Proof of the main result}\label{sec_proof_main_result}

The main idea of the proof of Theorem~\ref{thm_main} is to simplify the estimates introduced in \cite{KT2} and \cite{KTV} and localize the analysis to an arbitrarily small region near the free boundary $\Gamma_1$
and show that all important quantities can be controlled by
\eqnb\label{big_3}
\| \eta \|_{2.5+\delta }, \| \chi \eta \|_{3+\delta }, \| v \|_{2.5+\delta }
.
\eqne
First we employ the div-curl estimates to estimate each of the key quantities \eqref{big_3} at time $t$ using a time integral, from $0$ until $t$, of a polynomial expression involving the same quantities, as well as terms concerned with initial data and two other terms. The two terms involve derivatives of $\eta$ and $v$ only in the variables $x_1$, $x_2$, and only very close to the free boundary $\Gamma_1$, namely $\| S \eta_3 \|_{L^2 (\Gamma_1)}$ and $\| S (\psi v ) \|_{L^2}$, where $\psi$ is a cutoff supported in a neighborhood of $\Gamma_1$ with
$\supp \psi\subset \{\xi=1\}$; see \eqref{divcurl2}--\eqref{divcurl1} below for details.
The cutoff $\psi$ is introduced at the beginning of Section~\ref{sec_div_curl}.
These two terms, however, can also be controlled by a time integral of a polynomial expression involving \eqref{big_3} only, which we show in Section~\ref{sec_tangential}, after a brief discussion of some estimates, at each fixed time, of the pressure function and its time derivative in Sections \ref{sec_pressure}--\ref{sec_pressure_time_deriv}. Finally, Section~\ref{sec_final} combines the div-curl estimates with the tangential estimates to give an a priori bound that enables local-in-time existence and uniqueness.
\colb

Before proceeding to the proof, we note that, by \eqref{eta_def}, the particle map $\eta$ satisfies
\eqnb\label{na_eta_-_I}
\nabla \eta - I = \int_0^t \na v \d s,
\eqne
where, for brevity, we have omitted the time argument $t$ on the left-hand side. We continue this convention throughout.
Moreover, observing that $a(0)=I$, where $I$ denotes the three-dimensional identity matrix, we see from \eqref{a_exact_form} that
\begin{equation}
    a-I 
      = \int_0^t \p_t a   \d s 
      = \int_0^t \na \eta \na v \d s.
   \label{EQ01}
\end{equation}
Here and below, we use the convention of omitting writing various indices when only the product structure matters; for instance, the 
expression on the far right side stands for
$\epsilon_{imn} \epsilon_{jkl} \int_{0}^{t} \partial_{m}v_k\partial_{n}\eta_l$.
The equations \eqref{na_eta_-_I} and \eqref{EQ01} demonstrate an important property that, as long as the key quantities \eqref{big_3} stay bounded, both $a$ and $\na \eta$ remain close to $I$ in the $H^{1.5+\delta }$ norm for sufficiently small times. In other words we obtain the following lemma.
\begin{lemma}[Stability of $a$ and $\na \eta$ at initial time]\label{lem_stab_a_naeta}
Let $M,T_0>0$ and suppose that $\| v \|_{2.5+\delta }$, $\| \eta \|_{2.5+\delta }$, $\| \chi \eta \|_{3+\delta } \leq M$ for $t\in [0,T_0]$. Given $\varepsilon >0$, there exists $T=T(M,\varepsilon ) \in (0,T_0)$ such that
\eqnb\label{a-I_and_naeta-I}
\| I -a  \|_{1.5+\delta },\| I -aa^T  \|_{1.5+\delta }, \| I-\na \eta \|_{1.5+\delta } \leq \varepsilon
\eqne
for $t\in [0,T]$.
\end{lemma}
\begin{proof}
By \eqref{na_eta_-_I} and \eqref{EQ01}, we have $\| I-a \|_{1.5+\delta } \lec M^2 t $ and $\| \na\eta - I \|_{1.5+\delta } \lec M t$. Moreover, the triangle inequality and \eqref{a_in_H1.5} give $\| I-a a^T \|_{1.5+\delta } \leq \| I- a \|_{1.5+\delta } (1+ \| a \|_{1.5+\delta }) \lec M^2(1+M^2) t$, and so the claim follows by taking $T$ sufficiently small.
\end{proof}

Corollary \ref{cor_stab_rt} below provides a similar estimate for the pressure function, which enables us to extend the Rayleigh-Taylor condition \eqref{rayleigh_taylor} for small $t>0$.\\

\subsection{Div-curl estimates}\label{sec_div_curl}
Let $\psi (x_3) \in C^\infty (\RR ; [0,1] ) $ be such that $\supp \, \psi \subset \{ \chi =1 \}$ and $\psi =1$ in a neighborhood of $\Gamma_1$. Note that both $\chi$ and $\psi$ commute with any differential operator in the variables $x_1,x_2$, and that, provided $\psi$ is present in any given expression involving classical derivatives or $\Lambda$, we can insert in $\chi$ at any other place. For example
\eqnb\label{cutoffs_plugin}
\na f S g \na (\psi w) =  \na f S(\chi  g) \na (\psi w)   
\eqne
for any functions $f$, $g$, and $w$.

In this section, we provide estimates that allow us to control the key
quantities $\| v \|_{2.5+\delta }$, $\| \chi \eta \|_{3+\delta }$, and
$\| \eta \|_{2.5+\delta }$. Namely, we denote by $P$ any polynomial depending on these quantities, and we show that
\eqnb
\label{divcurl2}
\| \chi \eta \|_{3+\delta  } \lec t\| \chi \na \omega_0 \|_{1+\delta }+ 1+ \| \Lambda^{2.5+\delta } \eta_3 \|_{L^2 (\Gamma_1)}+\int_0^t P \,\d s,
\eqne
and 
\eqnb\label{v_2.5+delta}
\| v \|_{{2.5 +\delta }} \lec  \| v \|_{L^2} +\| S(\psi v) \|_{L^2}+\Vert u_0\Vert_{2.5+\delta}.
\eqne
Note that, on the other hand, by $\eta_t=v$ and $\eta(0,x)=x$, we have
\eqnb\label{divcurl1}
\|  \eta \|_{2.5+\delta  } \lec 1 + \int_0^t \Vert v\Vert_{2.5+\delta}\d s
\les 1+ \int_0^t P\d s
.
\eqne
As pointed out above,  we simplify the notation by omitting any indices in the cases where the exact value of the index becomes irrelevant. In those cases we only keep track of the power of the term and the order of any derivatives, as such terms are estimated using 
a H\"older, Sobolev, interpolation, or commutator inequality.

We start with the proof of \eqref{divcurl2}. We use \eqref{harmonic_est} to get 
\eqnb\label{001}
\| \chi \eta \|_{{3+\delta }} \lec \| \chi \eta \|_{L^2 } + \| \mathrm{curl}\, (\chi \eta ) \|_{{2+\delta }} + \| \mathrm{div}\, (\chi \eta ) \|_{{2+\delta }} + \| \Lambda^{2.5+\delta } \eta_3 \|_{L^2 (\Gamma_1)}.
\eqne
For the term involving curl, we first recall the Cauchy invariance
\eqnb\label{cauchy}
\epsilon_{ijk} \p_j v_m \p_k \eta_m = (\omega_0)_i,
\eqne
cf.~Appendix of \cite{KTV} for a proof.
For $i=1,2,3$, we have
\[\begin{split}
\nabla ((\mathrm{curl}\, (\chi \eta ))_i) &= \epsilon_{ijk}  \p_j \na (\chi \eta_k) \\
&= \epsilon_{ijk} \delta_{km} \p_j \na (\chi \eta_m ) \\
&= \epsilon_{ijk} (\delta_{km} -\p_k \eta_m ) \p_j \na (\chi \eta_m ) + 2 \chi \int_0^t \epsilon_{ijk}\p_k v_m \p_j \na \eta_m \, \d s + t \chi \na \omega_0^i \\
&\hspace{9cm}+ \underbrace{ \na \eta (D^2 \chi \eta + 2\na \chi \na \eta   )}_{=:LOT_1}   \\
&= \epsilon_{ijk} (\delta_{km} -\p_k \eta_m ) \p_j \na (\chi \eta_m ) + 2 \int_0^t \epsilon_{ijk}\p_k v_m \p_j \na (\chi \eta_m) \, \d s + t \chi \na \omega_0^i + LOT_1\\
&\hspace{7cm}+ \underbrace{\int_0^t  \na  v (D^2 \chi \eta + 2\na  \chi \na \eta +  \na \chi \na  \eta )  \d s,}_{=:LOT_2}   
\end{split}
\]
where we used 
\[
0= -\epsilon_{ijk} \p_k \eta_m \p_j \na \eta_m + 2 \int_0^t \epsilon_{ijk} \p_k v_m \p_j \na \eta_m \d s + t \na (\omega_0)_i
\]
in the third equality, which in turn is a consequence of the Cauchy invariance \eqref{cauchy}. Thus 
\eqnb\label{curl_chieta_est}
\begin{split}
 \| \nabla (\mathrm{curl}\, (\chi \eta )) \|_{{1+\delta }} &\lec \| \chi \eta \|_{{3+\delta}} \| I - \na \eta \|_{{1.5+\delta }} + 2\int_0^t  \underbrace{ \| v \|_{{2.5+\delta }} \|\chi  \eta \|_{{3+\delta }} }_{\leq P}\d s 
\\&
+ t\| \chi \na \omega_0 \|_{{1+\delta  }}+ \| LOT_1+LOT_2 \|_{{1+\delta }},
\end{split}
\eqne
where we used \eqref{kpv_est}. Note that
\[
\| LOT_1 \|_{1+\delta  } , \| LOT_2 \|_{1+\delta  } \lec 1+ \int_0^t \| v \|_{2.5+\delta } \| \eta \|_{1+ 1+\delta } \d s \lec 1+\int_0^t P\, \d s,
\]
where we used $\eta (x)=x+\int_0^t v (x,s)\d s$ 
and $T\leq 1/CM$ with $\Vert v\Vert_{2.5+\delta}\leq M$
to estimate $LOT_1$. 

As for the divergence, we have $\p_t \eta = v$, which gives
\[
\begin{split}
\div\, \p_l (\chi \eta ) &=  \delta_{kj} \p_l \p_k (\chi \eta_j ) =  ( \delta_{kj} - a_{kj} ) \p_l \p_k (\chi \eta_j ) + a_{kj} \p_l \p_k (\chi \eta_j ) \\
& =  ( \delta_{kj} - a_{kj} ) \p_l \p_k (\chi \eta_j ) + \int_0^t  \bigl( \p_t a_{kj} \p_l \p_k (\chi \eta_j ) +  a_{kj} \p_l \p_k ( \chi v_j )\bigr) \d s\\
& =  ( \delta_{kj} - a_{kj} ) \p_l \p_k (\chi \eta_j ) 
\\&
+ \int_0^t  \bigl( \p_t a_{kj} \p_l \p_k (\chi \eta_j ) +  \underbrace{a_{kj} ( \p_l \p_k \chi  v_j + \p_l \chi \p_k v_j + \p_k \chi \p_l v_j)}_{=:LOT_3}  - \chi \p_l a_{kj} \p_k v_j  \bigr) \d s,
\end{split}
\]
where in the last step we  used  $a_{kj} \p_l \p_k v_j = -\p_l a_{kj}  \p_k v_j$, a consequence of the divergence-free condition $a_{kj}\p_k v_j =0$. Therefore,
  \begin{align}
   \begin{split}
   \|\na  \div\, (\chi \eta ) \|_{{1+\delta }} 
   &\lec \| I - a \|_{{1.5+\delta }}  \| \chi \eta \|_{{3+\delta  }} 
   \\&
   + \int_0^t \bigl(  \underbrace{\| \p_t a \|_{{1.5+\delta/2 }} \| \chi \eta \|_{{3+\delta }}}_{\leq P}  + \| LOT_3 \|_{{1+\delta  }}  + \| \chi \p_l a \|_{{1+\delta}} \| v \|_{{2.5+\delta }} \bigr) \d s.
   \end{split}
  \notag
  \end{align}
Since $\p_l a $ consists of sums of the terms of the form $\p_a \eta_m \p_l \p_b \eta_n $, for $m,n,a,b=1,2,3$, we have 
\[ \| \chi \p_l a \|_{{1+\delta  }} \lec \| \eta \|_{{2.5+\delta }} (\| \chi \eta \|_{{3+\delta }} + \| \eta \|_{{2+\delta  }}) .\]
Moreover,
\[
\| LOT_3 \|_1+\delta \lec \| a \|_{1.5+\delta } \| v\|_{2+\delta } \lec P,
\] 
since $1+\delta < 1.5+\delta $. Thus 
\[
\| \na \div\, (\chi \eta ) \|_{{1+\delta }} \lec \| I - a \|_{{1.5+\delta }}  \| \chi \eta \|_{{3+\delta  }} + \int_0^t P\, \d s.
\]
Applying this inequality and \eqref{curl_chieta_est} into \eqref{001} gives
\[\begin{split}
\| \chi \eta \|_{{3+\delta }} &\lec  \| \chi \eta \|_{L^2} + \| \mathrm{curl}\, (\chi \eta ) \|_{1+\delta } + \| \mathrm{div}\, (\chi \eta ) \|_{1+\delta }   + \| \chi \eta \|_{{3+\delta }}\left( \| I-a \|_{1.5+\delta } + \| I - \na \eta \|_{1.5+\delta } \right)  \\
&+ \int_0^t P\, \d s  + t\| \chi \na \omega_0 \|_{1+\delta }+\Vert\Lambda^{2.5+\delta}\eta_3\Vert_{L^2(\Gamma_1)}  +1 .
\end{split}
\]
Recalling \eqref{a-I_and_naeta-I}, we may estimate the fourth term on the right-hand side by $\varepsilon \| \chi \eta \|_{3+\delta }$, and so we may absorb it on the left-hand side. Furthermore, we can absorb the second and the third terms on the right-hand side as well, by applying the Sobolev interpolation inequality between $L^2$ and $H^{3+\delta}$, which gives \eqref{divcurl2}.

As for the estimate \eqref{v_2.5+delta} for $v$, we note that the Cauchy invariance \eqref{cauchy} implies
\[
(\mathrm{curl}\, v )_i = \epsilon_{ijk} \p_j v_k = \epsilon_{ijk} (\delta_{km}- \p_k \eta_m )\p_j v_k + (\omega_0)_i,
\]
and the divergence-free condition, $a_{ji} \p_j v_i=0$, gives 
\[
\mathrm{div}\, v = \delta_{ji} \p_j v_i = (\delta_{ji}-a_{ji} ) \p_j v_i .
\]
Thus, using \eqref{harmonic_est}, we obtain
\eqnb\label{v_2.5+delta_temp} \begin{split}
\| v \|_{{2.5+\delta }} &\lec  \| v \|_{L^2} + \| \epsilon_{ijk} (\delta_{km} - a_{km } ) \p_j v_m  \|_{{1.5+\delta  }} 
\\&
+ \| (\delta_{ji} - a_{ji} ) \p_j v_i \|_{{1.5 +\delta }} +  \| \na_{2} v_3 \|_{H^{1+\delta  } (\Gamma_1)} + \| \omega_0 \|_{1.5+\delta  }
\end{split}
\eqne
For the boundary term, applying Sobolev interpolation and trace estimates gives 
\[
\| \na_{2} v_3 \|_{H^{1+\delta  } (\Gamma_1)} \lec \| v \|_{L^2 (\Gamma_1)} + \| \Lambda^{1.5 +\delta } v_3 \|_{H^{0.5}(\Gamma_1 )} \lec  \| v \|_{1} +  \| \Lambda^{1.5  +\delta} \na (\psi v_3 ) \|_{L^2}.
\]
As for the last term, since 
\begin{align}
   \begin{split}
   \p_3 (\psi v_3) &= \psi \div \, v - \psi \p_1 v_1 - \psi \p_2 v_2 + \p_3 \psi v_3,
    \\&
    =
    (\delta_{ji}-a_{ji})\p_j v_i
    - \p_1 (\psi v_1)
    - \p_2 (\psi v_2)
    + \p_3 \psi v_3
   \notag
   \end{split}
\end{align}
we have
\[\begin{split}
\| \Lambda^{1.5 +\delta} \na (\psi v_3 ) \|_{L^2} &\lec  \| \psi ( I-a ) \na v \|_{{1.5+\delta }} +  \| S (\psi v ) \|_{L^2} + \| \p_3 \psi v_3 \|_{{1.5 }} \\
&\lec \epsilon \| v \|_{{2.5+\delta }} +  \| v \|_{L^2} +  \| S(\psi v) \|_{L^2},
\end{split}
\]
where we used \eqref{a-I_and_naeta-I} and Sobolev interpolation of between $L^2$ and $H^{2.5 }$ for the last term. Applying this into \eqref{v_2.5+delta_temp}, and using \eqref{a-I_and_naeta-I} again we obtain 
\[
\| v \|_{{2.5+\delta }} \lec \epsilon \| v \|_{{2.5+\delta }} +  \| v \|_{L^2} + \| S(\psi v) \|_{L^2}+\| \omega_0 \|_{1.5+\delta },
\]
which, after absorbing the first term on the right-hand side, gives \eqref{v_2.5+delta}, as required.

\subsection{Pressure estimate}\label{sec_pressure}

In this section we show that if $\| q \|_{2.5+\delta }$ and $\| \psi q \|_{3+\delta }$ are finite, then
\eqnb\label{pressure_est_2.5}
\| q \|_{2.5+\delta } \lec \| v \|_{2+\delta }^2
\eqne
and 
\eqnb\label{pressure_est_with_psi}
 \| \psi q \|_{3+\delta } \leq P(\| \chi \eta \|_{3+\delta  }, \| v \|_{2.5+ \delta  }).
\eqne
In the remainder of this section we denote any polynomial of the form of the right-hand side by $P$, for simplicity. We also continue the convention of omitting the irrelevant indices.

We first prove \eqref{pressure_est_with_psi} assuming that \eqref{pressure_est_2.5} holds.
Multiplying the Euler equation,
$
\p_t v_i  = - a_{ki } \p_k q
$
by $\psi$ we obtain
$
\p_t (\psi v_i) = -a_{ki} \p_k (\psi q) + a_{ki} \p_k \psi \, q
$.
Now applying $a_{ji}\p_j$, summing over $i,j$, and using the Piola identity $\p_j a_{ji}=0$ we get
\[\begin{split}
&-\p_j (a_{ji} a_{ki} \p_k( \psi q)) + \p_j (a_{ji} a_{ki} \p_k \psi \, q )=a_{ji}\p_j \p_t (\psi v_i)
\\&\indeq
= \psi a_{ji}\p_j \p_t v_i + a_{ji} \p_j \psi \p_t v_i
= - \psi \p_t a_{ji}\p_j  v_i - a_{ji} \p_j \psi a_{ki} \p_k q ,
\end{split}
\]
where we have applied the product rule for $\p_j$ in the second equality, and, in the third equality, we used $\p_t (a_{ji} \p_j v_i )=0$ for the first term.
Thus
\eqnb\label{poisson_psiq}\begin{split}
\Delta (\psi q) &=  \p_j \bigl( (\delta_{jk} - a_{ji}a_{ki} ) \p_k (\psi q) \bigr) + \p_j (a_{ji} a_{ki}  \p_k (\psi q)) \\
&= \p_j \bigl((\delta_{jk} - a_{ji}a_{ki} ) \p_k (\psi q) \bigr) + \p_j (a_{ji} a_{ki}  \p_k \psi q) + \psi \p_t a_{ji} \p_jv_i  +  a_{ji} \p_j \psi a_{ki} \p_k q ,
\end{split}
\eqne
and thus, by the elliptic regularity,
noting that $\psi q=0$ on $\Gamma_0\cup \Gamma_1$,
\begin{align}\begin{split}
\| \psi q \|_{{3+\delta }} &\lec \| (I- aa^T )\na (\psi q) \|_{{2+\delta }} + \| a^2 \na \psi  q \|_{{2+\delta }} + \| \psi \p_t a \na v \|_{{1+\delta }} + \| a \na \psi a \na q \|_{{1+\delta }}\\
&\lec \| I- a a^T \|_{L^\infty } \| \psi q \|_{{3+\delta }} + \| \chi^4 (I-aa^T) \|_{{2+\delta }} \| \psi q \|_{{2.5+\delta }} \\
&+\| \chi^2 a \|_{2 +\delta }^2 \| q \|_{2+\delta }	+ \| \p_t a \|_{1.2+\delta } \| v \|_{2.3+\delta } + \| a \|_{1.5+\delta }^2 \| q \|_{2+\delta } \\
&\lec \varepsilon \| \psi q \|_{{3+\delta }} + (1+\| \chi \eta \|^4_{3+\delta  } ) \| \psi q \|_{2.5+\delta } + P,
\end{split}
\label{EQ02}
\end{align}
where we used \eqref{cutoffs_plugin}, the estimate $\| fg\|_{2+\delta } \lec \| f \|_{L^\infty } \| g \|_{2+\delta } + \| f \|_{2+\delta } \| g \|_{L^\infty }$, as well as the embedding $\| g \|_{L^\infty }\lec \| g \|_{1.5+\delta }$ and \eqref{kpv_est} in the second inequality. In the third inequality we used \eqref{at_in_H1.5} and \eqref{a-I_and_naeta-I}.
In the second term on the far right side of \eqref{EQ02}, we use \eqref{pressure_est_2.5}, proven below, to show that it is dominated by~$P$.
Thus we have obtained \eqref{pressure_est_with_psi}, given \eqref{pressure_est_2.5}, which we now verify.

In order to estimate $q$ in $H^{2.5+\delta }$, we note that, as in \eqref{poisson_psiq} above, $q$ satisfies the Poisson equation
\[
\p_{kk} q = \p_t a_{ji} \p_j v_i +  \p_j ( (\delta_{jk} - a_{ki}a_{ji} ) \p_k  q )
\]
in $\Omega$, together with homogeneous boundary condition $q=0 $ on $\Gamma_1$ (by \eqref{no_q_ontop}), and nonhomogeneous Neumann boundary condition $\p_3 q = (\delta_{k3} - a_{k3} )\p_k q$ on $\Gamma_0$, since taking $\p_t $ of the boundary condition $v_3 =0$ on $\Gamma_0$ gives $a_{k3} \p_k q =0$. Thus, elliptic estimates imply
\[\begin{split}
\| q \|_{2.5+\delta } &\lec \| \p_t a  \na v\|_{0.5+\delta } + \| (I-aa^T ) \na q \|_{1.5+\delta }+ \| (I-a) \na q \|_{H^{1+\delta }(\Gamma_0)}\\
&\lec \| \p_t a \|_{1+\delta } \| v \|_{2+\delta } + \| I-aa^T \|_{1.5+\delta} \| q \|_{2.5+\delta } + \| I-a \|_{1.5+\delta } \| q \|_{2.5+\delta }\\
&\lec \| v \|_{2+\delta }^2 + \varepsilon \| q \|_{2.5+\delta } ,
\end{split}
\]
where we used \eqref{at_in_H1.5} and \eqref{a-I_and_naeta-I} in the last step. Taking a sufficiently small $\varepsilon >0$ proves \eqref{pressure_est_2.5}, as required.
\subsection{Time derivative of $q$}\label{sec_pressure_time_deriv}
In this section, we supplement the pressure estimates \eqref{pressure_est_2.5} and \eqref{pressure_est_with_psi} from the previous section with
\eqnb\label{dt_q}
\| \p_t q \|_{2+\delta } \leq P,
\eqne
\eqnb\label{dt_psi_q_2.5}
\| \p_t (\psi  q) \|_{2.5+\delta } \leq P,
\eqne
where $P$ denotes a polynomial in terms of 
$\| \chi \eta \|_{3+\delta }$, $\| \eta \|_{2.5+\delta }$, and $\| v \|_{2.5+\delta } $.
We note that \eqref{dt_psi_q_2.5} implies that the Rayleigh-Taylor condition \eqref{rayleigh_taylor} holds for sufficiently small $t$ given our key quantities \eqref{big_3} remain bounded.
\begin{cor}[Rayleigh-Taylor condition for small times]\label{cor_stab_rt}
Let $M,T_0>0$, and suppose that $\| v \|_{2.5+\delta }$, $\| \eta \|_{2.5+\delta }$, $\| \chi \eta \|_{3+\delta } \leq M$ for $t\in [0,T_0]$. There exists $T=T(M,b ) \in (0,T_0)$ such that
\eqnb\label{rayleigh_t_small_times}
\p_3 q (x,t) \leq - \frac{b}2 
\eqne
for $x\in \Gamma_1$ and $t\in [0,T]$.
\end{cor}
\begin{proof}
The proof is analogous to the proof of Lemma~\ref{lem_stab_a_naeta}, by noting that for $x\in \Gamma_1$
\[
\p_3 q (x,t) - \p_3 q(x,0) = \int_0^t \p_t \p_3 q \d s \leq \int_0^t \|\p_t \na q \|_{L^\infty (\Gamma_1) } \d s \leq \int_0^t \| \p_t (q \psi) \|_{{2.5+\delta }} \d s\leq P(M) t, 
\]
which is bounded by $b/2$ for sufficiently small $t$.
\end{proof}

We prove \eqref{dt_psi_q_2.5} first. Applying $\p_t$ to \eqref{poisson_psiq} we obtain
\eqnb\notag\begin{split}
\Delta \p_t(\psi q) &= \p_j \bigl( (\delta_{jk} - a_{ji}a_{ki} ) \p_k \p_t(\psi q) \bigr)  -2\na (a\p_t a  \psi \na q ) + \na (2a\p_t a \na \psi q + a^2 \na \psi \p_t q)\\
&+\psi \p_{tt} a \na v +\psi \p_t a \na \p_t v  +2a\p_t a \na \psi \na q  + a^2 \na \psi \na \p_t q.
\end{split}
\eqne
Noting that $\p_t (\psi q)$ satisfies the homogeneous Dirichlet boundary conditions at both $\Gamma_1$ and $\Gamma_0$, and is periodic in $x_1$ and $x_3$, the elliptic regularity gives,
for $\sigma=0.5+\delta$,
\begin{align}\begin{split}
\| \p_t (\psi q ) \|_{2 + \sigma  } &\lec \| (I-aa^T ) \na \p_t (\psi q ) \|_{1+\sigma } +\| a \p_t a \psi \na q  \|_{1+\sigma  }
+ \| a \p_t a \na \psi q \|_{1+\sigma} 
\\& 
+ \| a^2 \na \psi \p_t q \|_{1+\sigma }   
+ \|  \p_{tt} a \na v \|_{\sigma  }
+ \| \p_t a \na (a \na q) \|_{\sigma  } 
+ \| a \p_t a \na q \|_{\sigma  } 
+ \| a^2 \na \p_t q \|_{\sigma  }
\\&
\lec 
\varepsilon \| \p_t (\psi q ) \|_{2+\sigma  } 
+ \|   a \|_{1.5 +\delta } \| \p_t a \|_{1.5 +\delta   } \| q \|_{2+\sigma  } 
+\| a \|_{1.5+\delta }^2 \| \na \psi \p_t q \|_{1+\sigma }
\\&
+ \| \p_{tt} a \|_{\sigma } \| v \|_{2.5 +\delta}
+ \| \p_t a \|_{1.5 +\delta } \| a \|_{1.5 +\delta  } \| q \|_{2+\sigma } 
\\&
+ \| a \|_{1.5 +\delta } \| \p_t a \|_{1.5+\delta } \| q \|_{1+\sigma  } 
+  \| a \|_{1.5 +\delta}^2 \| \na \psi \na \p_t q \|_{\sigma  }  ,
\end{split}
\label{EQ03}
\end{align}
where we used \eqref{a-I_and_naeta-I} in the second inequality.
 Taking $\varepsilon >0 $ sufficiently small, and absorbing the first term on the left-hand side, and using \eqref{pressure_est_2.5}, \eqref{at_in_H1.5}, and \eqref{att_in_Hr}, we obtain
\eqnb\label{dt_q_naive}
\| \p_t (\psi q ) \|_{2 + \sigma  } \lec P (1+ \| \na \psi \p_t q \|_{1+\sigma} +\| \na \p_t q \|_{\sigma  }  ).
\eqne
Thus, given \eqref{dt_q}, we have obtained \eqref{dt_psi_q_2.5}.

It remains to show \eqref{dt_q}. Although it might seem that taking $\psi \coloneqq 1$ and $\sigma \coloneqq \delta $ in \eqref{dt_q_naive} might seem equivalent to \eqref{dt_q}, we must note that in such case $\p_t q$ does not satisfy the homogeneous Dirichlet boundary condition at $\Gamma_0$. Instead, we have the Neumann condition 
\[\p_3 \p_t q = (\delta_{k3} - a_{k3}) \p_k q - \p_t a_{k3} \p_k q,
\]
by applying $\p_t$ to $a_{k3}\p_k q =0$ on $\Gamma_1$. Hence, in the case $\psi =1$ and $\sigma \coloneqq \delta$, the estimate \eqref{dt_q_naive} does include the last two terms inside the parantheses, as $\na \psi = \na 1=0$, but instead  \eqref{EQ03} needs to be amended by the boundary term
\[
\| (I-a ) \na \p_t q - \p_t a \na q \|_{H^{0.5+\delta  }(\Gamma_0 )} \lec \| I- a \|_{1.5+\delta } \| \p_t q\|_{2+\delta } + \| \p_t a \|_{1.5+\delta } \| \na q \|_{2+\delta } \lec \varepsilon \| \p_t q\|_{2+\delta } + P,
\] 
where we used \eqref{a-I_and_naeta-I} again and \eqref{at_in_H1.5}, \eqref{pressure_est_2.5}. Thus choosing a sufficiently small $\varepsilon >0$ and absorbing the first term on the right-hand side we obtain \eqref{dt_q}, as required.

\subsection{Tangential estimates}\label{sec_tangential}
In this section we show that 
\eqnb\label{tang_est}
\| S (\psi v ) \|_{L^2}^2 + \| a_{3l} S \eta_l \|_{L^2 (\Gamma_1 )}^2  \lec  \| \psi v_0 \|_{2.5+\delta }^2 +  \| v_0 \|_{2+\delta }^2+\int_0^t  P\, \d s,
\eqne
where, as above, $P$ denotes a polynomial in 
$\| v (s) \|_{{2.5+\delta }}$, $\| \chi \eta (s) \|_{{3+\delta }}$, and $\| \eta  \|_{2.5+\delta }$.

Our estimate follows a similar scheme as \cite[Lemma~6.1]{KT2}, except that the argument is shorter and sharper in the sense that we eliminate the dependence on $\p_t v$, $q$ and $\p_t q$. 
Another essential change results from the appearance of the cutoff $\psi$. This localizes the estimate and causes minor changes to the scheme. Note that the appearance of $\psi$ is essential for localizing the highest order dependence on $\eta$. Namely, it allows us to use the $H^{3+\delta }$ norm of merely $\chi \eta$, rather than $\eta$ itself, which we only need to control in the $H^{2.5+\delta }$ norm.

In the remainder of Section~\ref{sec_tangential}, we prove~\eqref{tang_est}.
Note that $S\p_t (\psi v ) = -S( a_{ki} \psi  \p_k q ) $, which gives 
\[
\frac{\d }{\d t} \| S(\psi v) \|_{L^2}^2 = -\int S (a_{ki} \psi \p_k  q )S(\psi v_i) .
\]
In what follows, we estimate the right-hand side by $P+I_{113}$, where $I_{113}$ 
is defined in \eqref{EQ04} below and
is such that
\eqnb\label{I113_toshow}
\int_0^t I_{113} \, \d s \lec - \| a_{3l} S \eta_l (t) \|_{L^2 (\Gamma )} + \| \p_3 q (0)  \|_{1.5+\delta } + \int_0^t P \,\d s.
\eqne
The claim then follows by integration in time on $(0,t)$ and by recalling \eqref{pressure_est_2.5}.
First we write
\begin{align}\begin{split}
-\int  S(a_{ki}  \psi \p_k q ) S (\psi v_i) &=- \int  Sa_{ki}  \psi \p_k q \, S (\psi v_i)-\int  a_{ki}    \p_k S(\psi q) \, S (\psi v_i)+\int  a S ( \na \psi  q) S (\psi v) \\
&- \int  \bigl( S(a  \psi \na q )-Sa  \psi \na q   - a S(\psi \nabla q)  \bigr) S (\psi v ) \\
&=: I_1 + I_2+I_3 +I_4,
\end{split}
\notag
\end{align}
where we used ${\psi \p_k q}{=\p_k (\psi q ) - \p_k \psi q}$ to obtain the second and third terms. Using \eqref{pressure_est_2.5} we obtain  
\[
I_3 \lec \| a \|_{L^\infty } \|  q \|_{2.5+\delta  } \| \psi v \|_{2.5 +\delta } \leq P.
\]
For $I_2$, we integrate by parts in $x_k$ and use the Piola identity $\p_k a_{ki}=0$ to get
\[\begin{split}
I_2 &=  \int a_{ki}  S(\psi q) \p_k S (\psi v_i)=\int a_{ki} S(\psi q)  S (\psi \p_k v_i)+\underbrace{ \int a S(\psi q)  S (\na  \psi v)}_{\lec \| a \|_{L^\infty } \| q \|_{2.5 +\delta } \| v \|_{2.5 +\delta }} .
\end{split}
\]
Note that the boundary terms in the first step vanish, as $q=0$ on $\Gamma_1$ and $\psi$ vanishes in a neighborhood of $\Gamma_0$.
Moving $S$ away from $\psi \p_k v$ in the first term above, and recalling the divergence-free condition $a_{ki} \p_k v_i =0$, we obtain 
\[ a_{ki} S (\psi \p_k v_i ) = - S a_{ki} \psi \p_k v_i - (S(a_{ki} \psi \p_k v_i ) - a_{ki} S (\psi \p_k v_i ) - S a_{ki} \psi \p_k v_i ) ,
\]
and thus
\[\begin{split}
I_2& \lec - \int S a S(\psi q)  \psi \na  v    +\| S(\psi q) \|_{L^3 } \| S(a \psi \na v) - a S(\psi \na v) - Sa \psi \na v \|_{L^{\frac32}} +P \\
&\lec  - \int \Lambda^{2+\delta  } a \Lambda^{\frac12}\left( S(\psi q)  \psi \na v \right)   + \| \psi q \|_{3+\delta }( \| a \|_{W^{1,6}} \|  v \|_{{2.5+\delta }} + \| a \|_{W^{1.5+\delta , 3}}  \| v \|_{W^{2,3}} ) +P\\
&\lec \| \chi^2 a \|_{2+\delta  } \| S (\psi q ) \psi \na v \|_{0.5}+P
\lec \| \chi \eta \|_{3 +\delta }^2 \| \psi q \|_{3+\delta }  \|  v \|_{2.5+\delta }  +P \lec P ,
\end{split}
\]
where we used the embedding $H^{0.5} \subset L^3$, the commutator estimate \eqref{kpv3} in the second inequality, and the embeddings $H^1 \subset L^6$, $H^{0.5}\subset L^3$; we also used \eqref{pressure_est_with_psi} and \eqref{achi_in_H2}  in the fourth and fifth inequalities respectively.

As for $I_4$ we have
\[\begin{split}
I_4 &\lec \| \psi v \|_{2.5+\delta } \| S (a  \psi \na q ) - S a \psi \na q - a S (\psi \na q) \|_{L^2} \\
&\lec P \| S (a  \psi \na q ) - S a \psi \na q - a S (\psi \na q ) \|_{L^2}\\
&\lec P (\| \chi a \|_{W^{1,6}} \| \psi \na q \|_{W^{1.5+\delta , 3}} + \| \chi a \|_{W^{1.5+\delta ,3}} \| \psi \na q \|_{W^{1,6}} )\lec P,
\end{split}
\]
where we have applied \eqref{kpv3}, \eqref{achi_in_H2}, 
\eqref{cutoffs_plugin}, \eqref{pressure_est_2.5}, and \eqref{pressure_est_with_psi} in the last line. 

For $I_1$, we write
\[
S = \sum_{m=1,2} S_m \p_m + S_0,
\] 
where $S_m \coloneqq  -\Lambda^{\frac12+{\delta }} \p_m$ and $S_0 \coloneqq \Lambda^{\frac12+{\delta }}$. 
Furthermore, differentiating the identity $a\na \eta =I$ with respect to $x_m$, where $m=1,2$, we get $\p_m a \na \eta =- a\p_m \na \eta $, and then multiplying this matrix equation by $a$ on the right-side gives $\p_m a = - a \p_m \na \eta a$, which in components reads
\[
\p_m a_{ki} = - a_{kj} \p_m \p_l \eta_j a_{li}
.
\]
Thus we may write 
\[
Sa_{ki} = - S_m ( a_{kj} \p_m \p_l \eta_j a_{li} )  + S_0 a_{ki},
\]
and consequently
\[\begin{split}
I_1 &= \sum_{m=1,2}\int S_m ( a_{kj} \p_m \p_l \eta_j a_{li} )  \psi \p_k q\, S(\psi v_i) - \int S_0 a  \psi \na q S(\psi v) \\
&=   \sum_{m=1,2}\int  a_{kj} S_m \p_m   \p_l \eta_j a_{li}   \psi \p_k q S(\psi v_i) +  \sum_{m=1,2}\int 
\bigl( S_m ( a 
D^2\eta a ) -  a S_m D^2  \eta a  \bigr)   \psi \na q S(\psi v) \\&
- \int S_0 a \psi \na q S(\psi v) \\
&=: I_{11} + I_{12} + I_{13} 
.
\end{split} \]
We have 
\[\begin{split}
I_{13} &\lec \| S_0 a \|_{L^2} \| \na q \|_{L^\infty } \| S ( \psi v ) \|_{L^2} \lec \|a \|_{0.5+\delta } \| q \|_{2.5+\delta  } \| v \|_{2.5+\delta } \leq P,
\end{split}
\]
and 
\[\begin{split}
I_{12} &\lec \| \na q \|_{L^\infty } \| \psi v \|_{2.5+\delta } \| S_m ((\chi^2 a)^2 D^2 ( \chi \eta) ) - (\chi^2 a)^2 S_m D^2 (\chi \eta ) \|_{L^2}  \\
&\lec P \left( \| (\chi^2 a)^2 \|_{W^{1.5+\delta , 3}}  \| \chi \eta \|_{W^{2,6}} + \| (\chi^2 a )^2 \|_{W^{1,6}} \| \chi \eta \|_{W^{2.5+\delta, 3}} \right) \lec P, 
\end{split}
\]
as claimed, where we used \eqref{cutoffs_plugin} in the first inequality, and \eqref{kpv1}, \eqref{achi_in_H2} in the second.

For $I_{11}$ we write $\sum_{m=1,2}S_m \p_m =  - S_0+S $, integrate by parts in $x_l$ and use the Piola identity $\p_l a_{li} =0$ to get
\begin{align}\begin{split}
I_{11} & = -\int  a S_0  \na \eta  a \psi \na q S(\psi v) -\int \na a S    \eta a   \psi \na q S(\psi v)-\int  a S    \eta a  \na \psi \na q S(\psi v)  \\
&\underbrace{-\int  a S    \eta a    \psi D^2 q S(\psi v)}_{=:I_{111}}\underbrace{-\int  a_{kj} S    \eta_j a_{li}   \psi\p_k q \p_l S(\psi v_i)}_{=:I_{112}}+\underbrace{ \int_{\Gamma_1} a_{kj} S \eta_j a_{3i} \p_k q S(\psi v_i ) \d \sigma}_{=:I_{113}}\\
&\lec \| a \|^2_{L^\infty } \| q \|_{W^{1,\infty }} \| \eta \|_{2.5+\delta } \|v \|_{2.5+\delta } 
\\&
+ \|\na (\chi^2a ) \|_{L^6 } \| S(\chi \eta ) \|_{L^3} \| a \|_{L^\infty } \| q \|_{W^{1,\infty }}\| v\|_{2.5+\delta } + I_{111}  + I_{112}+I_{113} \\
&\lec P + I_{111}  + I_{112}+I_{113}
\end{split}
   \label{EQ04}
\end{align}
where we used \eqref{cutoffs_plugin} in the first inequality, and inequalities $\| f \|_{L^\infty }\lec \| f \|_{1.5+\delta }$, and \eqref{achi_in_H2} to obtain $\| \na (\chi^2 a )\|_{L^6} \lec \| \chi^2 a \|_{2} \lec \| \chi \eta \|_{3+\delta }$ in the second inequality. Note that there is no boundary term at $\Gamma_0$ as $\psi =0$ on~$\Gamma_0$.

For $I_{111}$ we write $\psi D^2 q = D^2 (\psi q ) - 2 \na \psi \na q - D^2 \psi q$ to obtain
\[
\begin{split}
I_{111} &= -\int  a S    \eta a    D^2 (\psi q) S(\psi v)+\int  a S    \eta a    (2\na \psi \na q + D^2  \psi  q ) S(\psi v) \\
&\lec \| a \|_{L^\infty}^2 \| S(\chi \eta ) \|_{L^3} \| v\|_{2.5+\delta } ( \| \psi q \|_{W^{2,6}} + \| q \|_{W^{1,6 }} )\lec P,
\end{split}
\]
where we used \eqref{achi_in_H2}, \eqref{pressure_est_with_psi},
and the embedding $H^{0.5}\subset L^3$
in the last inequality.
As for $I_{112}$, we note that the divergence-free condition, $a_{li} \p_l v_i=0$, gives 
\[S (a_{li} \p_l (\psi v_{i}))= S (a_{li} \p_l \psi v_{i}). \] In order to use this fact, we put $a_{li}$ inside the second $S$ in $I_{112}$ and extract the resulting commutator. Namely, we denote $f\coloneqq -a_{kj} S\eta_j  \psi\p_k q$ for brevity, and write
\begin{align}
\begin{split}
I_{112} &= \int  a_{li}  S\p_l (\psi v_i )f\\
&=  \int  S (a_{li}  \p_l \psi v_i)f -\int  S a    \na (\psi v)f\underbrace{-\int   \bigl( S(a \na  (\psi v )) - S a \na  (\psi v) -a S\na  (\psi v) \bigr)f}_{\lec  \| S (a \na (\psi v)) - Sa \,\na (\psi v) - a S\na (\psi v) \|_{L^2} \| f \|_{L^2}}\\
&\lec     \int  \bigl( S(a \na  \psi v) - Sa  \na  \psi v- a S( \na  \psi v) \bigr)f -\underbrace{\int   S a   \psi \na  v f}_{=\int \Lambda^{2+\delta }  (\chi^2 a ) \Lambda^{\frac12} (\psi \na v f ) } \hspace{0.5cm}+\underbrace{ \int  a S (\na  \psi  v)f}_{\lec  \| a \|_{L^\infty }  \| v \|_{2.5+\delta }\|f \|_{L^2}} \\
&+ \| S (a \na (\psi v)) - Sa \,\na (\psi v) - a S\na (\psi v) \|_{L^2}\| f \|_{L^2}\\
&\les \| \chi^2 a \|_{2+\delta } \| \Lambda^{\frac12 } ( \psi \na v f) \|_{L^2} \\
&+  \| f \|_{L^2} \bigl( P+\| S(a \na \psi v ) - Sa \na \psi v- a S(\na \psi v) \|_{L^2} + \| S (a \na (\psi v)) - Sa \,\na (\psi v) - a S\na (\psi v) \|_{L^2} \bigr) \\
&\lec P,
 \end{split}
\notag
\end{align}
where we have used \eqref{cutoffs_plugin} in the first inequality. In the last step we have applied \eqref{achi_in_H2}
and used $\| f \|_{L^2} \lec \| a \|_{L^\infty } \| \chi \eta \|_{2.5+\delta } \| q \|_{W^{1,\infty }} \lec P$; we have also estimated both commutator terms by
\[
\| \chi^2 a \|_{W^{1, 3 }} \| v \|_{W^{1.5+\delta , 6}} + \| v \|_{W^{1,6}} \| \chi^2 a \|_{W^{1.5+\delta,3 }}   \leq \| \chi^2 a \|_{{2+\delta }} \| v \|_{{2.5+\delta}} \leq P,
\]
using the Kato-Ponce inequality \eqref{kpv3} and \eqref{achi_in_H2}, \eqref{cutoffs_plugin}, and we estimated the remaining factor by
\[\begin{split}
   \| \Lambda^{\frac12 } ( \psi \na v f) \|_{L^2} 
     &\lec
       \Vert \Lambda^{\frac12}f\Vert_{L^2}
       \Vert \psi\nabla v\Vert_{L^\infty}
       +
       \Vert f\Vert_{L^3}
       \Vert \Lambda^{\frac12}(\psi\nabla v)\Vert_{L^6}
       \\&
       \lec
       \Vert f\Vert_{0.5}
       \Vert v\Vert_{W^{1,\infty}}
       +
       \Vert f\Vert_{0.5}       
       \Vert v\Vert_{2.5}
       \\&
       \lec 
       \| a \|_{1.5+\delta } \| \chi \eta \|_{3+\delta } \| q \|_{1.5+\delta } 
       \Vert v\Vert_{2.5+\delta}
       \leq P.
     \end{split}
     \]
\colb

It remains to estimate the boundary term $I_{113}$, as claimed in \eqref{I113_toshow}.
We note that $\p_k q=0$ for $k=1,2$ on $\Gamma_1$, and $v_i = \p_t \eta_i$, which gives $a_{3i} Sv_i = \p_t (a_{3i} S \eta_i) - \p_t a_{3i} S\eta_i$. Thus
\[\begin{split}
I_{113}& = \frac12 \frac{\d }{\d t} \int_{\Gamma_1} |a_{3i} S \eta_i |^2 \p_3 q \d \sigma - \int_{\Gamma_1} a_{3j} S \eta_j \p_3 q \p_t a_{3i} S \eta_i \d \sigma -  \int_{\Gamma_1} a_{3j} S \eta_j \p_3 \p_t q  a_{3i} S \eta_i \d \sigma  \\
&\leq \frac12 \frac{\d }{\d t} \int_{\Gamma_1} |a_{3i} S \eta_i |^2 \p_3 q \d \sigma + \| a \|_{L^\infty } \| \p_t a \|_{L^\infty } \| q \|_{W^{1,\infty }} \| S(\chi \eta) \|_{L^2 (\Gamma_1)} + \| a \|_{L^\infty}^2 \| S (\chi \eta )\|_{L^2(\Gamma_1 )} \| \p_t (\psi q ) \|_{W^{1,\infty }}  \\
&\leq \frac12 \frac{\d }{\d t} \int_{\Gamma_1} |a_{3i} S \eta_i |^2 \p_3 q \d \sigma + P,
\end{split}
\]
where we used \eqref{a_in_H1.5}, \eqref{at_in_H1.5}, the facts $\| S(\chi \eta ) \|_{L^2 (\Gamma_1)} \lec \| \chi \eta \|_{3+\delta } \leq P$, and $\| q \|_{W^{1,\infty }} \lec \| q \|_{2.5+\delta }\leq P$, as well as the pressure estimate \eqref{dt_psi_q_2.5} to get $\| \p_t (\psi q ) \|_{W^{1,\infty }} \lec \| \p_t (\psi q ) \|_{2.5+\delta } \leq P $.

 Consequently, the Rayleigh-Taylor condition for small times \eqref{rayleigh_t_small_times} and the fact that $a_{ki}=\delta_{ki}$ at time $0$ give 
\[\begin{split}
\int_0^t I_{113}\, \d s &\leq \frac12 \left. \int_{\Gamma_1} |a_{3i} S \eta_i |^2 \p_3 q \d \sigma \right|_{t} - \frac12 \left. \int_{\Gamma_1} | S \eta |^2 \p_3 q \d \sigma \right|_0  + \int_0^t P \, \d s\\
&\lec -\| a_{3i } S \eta_i \|_{L^2}^2 + \| \p_3 q(0) \|_{L^\infty } + \int_0^t P \, \d s,
\end{split} 
\]
where we used $\eta(x) =x$ at time $0$ in the last step.
This concludes the proof of \eqref{tang_est}.

\subsection{Final estimates}\label{sec_final}
In this section, we collect the estimates thus completing the proof of the main theorem.

\begin{proof}[Proof of Theorem~\ref{thm_main}]
From the inequalities \eqref{v_2.5+delta}, \eqref{divcurl1}, and \eqref{divcurl2},
combined with the tangential estimate \eqref{tang_est}, give the inequality
\eqnb\label{apriori_est}
\| v \|_{2.5+\delta }, \| \eta \|_{2.5+\delta }, \| \chi \eta \|_{3+\delta } \lec \int_0^t P\, \d s + \|  \psi v_0 \|_{2.5+\delta }^2 + \| v_0 \|_{2+\delta }^2 + t \| \chi \omega_0 \|_{2+\delta } + \| \omega_0 \|_{1.5+\delta } +1 + \|v_0 \|_{L^2},
\eqne
where $P$ is a polynomial in $\| v \|_{2.5+\delta }, \| \eta \|_{2.5+\delta }$ and $\| \chi \eta \|_{3+\delta }$. Note that, in order to estimate the boundary terms on the right hand side in \eqref{divcurl1}, \eqref{divcurl2}, we have used
\[
\| S\eta_3 \|_{L^2 (\Gamma_1 )} \leq \| a_{3l} S\eta_l \|_{L^2 (\Gamma_1 )} + \| (\delta_{3l} - a_{3l}) S\eta_l \|_{L^2 (\Gamma_1 )} \leq \| a_{3l} S\eta_l \|_{L^2 (\Gamma_1 )} + \varepsilon \| \chi \eta  \|_{3+\delta } ,
\]
where we applied \eqref{a-I_and_naeta-I} and a trace estimate in the second step, and we absorbed the last term by the left hand side above. Moreover, in order to obtain the initial kinetic energy $\| v_0 \|_{L^2}$ in \eqref{apriori_est}, instead of $\| v \|_{L^2}$ from \eqref{v_2.5+delta}, we note that

\[
\| v \|_{L^2} \leq \| v_0 \|_{L^2} + \int_0^t \| \p_t v\|_{L^2} \d s \leq \| v_0 \|_{L^2} + \int_0^t \| a \|_{1.5+\delta } \| q \|_{1} \d s \lec  \| v_0 \|_{L^2} + \int_0^t P \, \d s, 
\]
where we used $\partial_{t}=-a\nabla q$, in the second inequality
and \eqref{a_in_H1.5}, \eqref{pressure_est_2.5} in the last.

The a priori estimate \eqref{apriori_est} allows us to apply a standard Gronwall argument, concluding the proof.
\end{proof}

\end{document}